\documentclass[11pt]{article}\usepackage[utf8]{inputenc}
\usepackage{amsthm, amsmath, amssymb, amsfonts, url, booktabs, tikz, setspace, fancyhdr, bm}
\usepackage[colorlinks=true]{hyperref}
\usepackage[margin = 1in]{geometry}
\usepackage{enumerate}
\usepackage[shortlabels]{enumitem}
\usepackage[babel]{microtype}
\usepackage[english]{babel} 
\usepackage[capitalise]{cleveref}
\usepackage{comment}
\usepackage{scalerel}
\usetikzlibrary{arrows.meta}

\newtheorem{theorem}{Theorem}[section]

\newtheorem{lemma}[theorem]{Lemma}
\newtheorem{corollary}[theorem]{Corollary}

\newtheorem{claim}[theorem]{Claim}

\theoremstyle{definition}
\newtheorem{definition}[theorem]{Definition}

\newtheorem{question}[theorem]{Question}

\theoremstyle{remark}

\newcommand{\defeq}{:=}
\newcommand{\one}{\mathbf{1}}
\newcommand{\E}[1]{\mathbf{E}\left[#1\right]}
\newcommand{\Var}[1]{\mathbf{Var}\left[#1\right]}
\newcommand{\Cov}[1]{\mathbf{Cov}\left[#1\right]}
\newcommand{\dif}{D}

\newcommand{\eps}{\varepsilon}
\newcommand{\adj}{\sim}
\newcommand{\Bin}{\operatorname{Bin}}

\newcommand{\adv}{\Delta}

\newcommand{\Clr}{\ell}
\newcommand{\bClr}{\widehat{\ell}}
\newcommand{\bR}{\widehat{R}}
\newcommand{\bB}{\widehat{B}}

\newcommand{\bdif}{\widehat{\dif}}

\newcommand{\R}{\mathbb{R}}

\newcommand{\N}{\mathbb{N}}

\renewcommand{\epsilon}{\eps}
\renewcommand{\phi}{\varphi}
\renewcommand{\iff}{\Leftrightarrow}
\renewcommand{\Pr}{{\mathbf P}}

\title{A new density limit for unanimity in majority dynamics on random graphs}
\author{Jeong Han Kim and BaoLinh Tran}
\date{June 2022}

\begin{document}

\maketitle

\begin{abstract}
Majority dynamics is a process on a simple, undirected graph $G$ with an initial Red/Blue color for every vertex of $G$.
Each day, each vertex updates its color following the majority among its neighbors, using its previous color for tie-breaking.
The dynamics 
achieves {\em unanimity}
if every vertex has the same color after finitely many days, and such color is said to \emph{win}.

When $G$ is a $G(n,p)$ random graph, L. Tran and Vu (2019) found a codition
in terms of $p$ and the initial difference $2\adv$ beteween the sizes of the Red and Blue camps, such that unanimity is achieved with probability arbitrarily close to 1.
They showed that if $p\adv^2 \gg1 $, $p\adv \geq 100$, and $p\geq (1+\eps) n^{-1}\log n$ for a positive constant $\eps$, then unanimity occurs with probability $1 - o(1)$.
If $p$ is not extremely small, namely $p > \log^{-1/16} n $, then Sah and Sawhney (2022) showed that the condition $p
\adv^2 \gg  1$ is sufficient. 

If 
$n^{-1}\log^2 n \ll p \ll n^{-1/2}\log^{1/4} n, $
we show that $p^{3/2}\adv \gg  n^{-1/2}\log n$ 
is enough. 
Since this condition holds if $p\adv \geq 100$ for $p$ in this range, this is an improvement of Tran's and Vu's result. 
For the closely related problem of finding the optimal condition for $p$ to achieve unanimity when the initial coloring is chosen uniformly at random among all possible Red/Blue assignments, our result implies a new lower bound $p \gg n^{-2/3}\log^{2/3} n$, which improves upon the previous bound of $n^{-3/5}
\log n$ by 
Chakraborti, Kim, Lee and T. Tran (2021).
\end{abstract}

\section{Introduction}

\subsection{Majority Dynamics}

Consider a parliamentary election with two parties.
The election season begins with Day 0, where each voter supports one side.
At the end of every day, each voter polls their friends, then supports the majority side the next morning, or keeps their affiliation in case of a tie.
On election day, everyone votes for the party they currently supports and the parliament will be divided proportionally with each side's number of votes.

This process can be modeled with a graph process, where vertices represent voters and edges represent friendships.
We let Red and Blue, or R and B, stand for the two parties.
The state of affiliations at any time is simply a coloring on the graph.
This process on a graph is called \textbf{Majority Dynamics} in literature, due to the majority-based updating rule.

There is a vast literature of more general \emph{opinion exchange} processes, with applications in many fields.
We refer to the influential survey by Mossel and Tamuz \cite{mosselTamuz2017} for an overview.
In this paper, we focus on Majority Dyanmics and study its final states.

While previous works \cite{fountoulakis, jhk-etal2021} used $\pm 1$ assignments in place of Red and Blue, we adopt the Red/Blue notation to be consistent with a recent paper by the second author and Vu \cite{TranVusparse}, which uses the following formal definition

\begin{definition}[Majority Dynamics]  \label{defn:majority-dynamics}
    Given simple undirected graph $G = (V, E)$, and a partition $V = R\cup B$, the \textbf{majority dynamics} on $G$ with initial sides $(R, B)$ is a process that begins with Day 0, where vertices in $R$ and $B$ are respectively colored Red and Blue.
    
    For each $t \ge 0$, on day $t + 1$,
    each vertex adopts the color the majority of its neighbors hold on day $t$, or remains unchanged in case of a tie.
\end{definition}



On theory, the process goes on forever, but a beautiful argument, independently discovered by Goles and Olivos \cite{goles1981} and Poljak and Sura \cite{poljak1983}, show that
after at most $|V|$ days,
the process either stabilizes or alternates forever between two colorings, so we can view them as the \emph{de facto} final states.
In the settings in this paper, the process will almost always ends in a special stable position called \textbf{unanimity}, where every vertex has the same color.
There is a large and active research area dedicated to finding the conditions under which unanimity occur.
In the next section we will introduce two main models in the literature, their relationship, and the main questions in each.

Let us first introduce extra notations to aid the presentation.
\begin{itemize}

    \item Given a coloring $(R, B)$ of $V$, there is a corresponding map $\Clr(R, B): V \to \{\pm 1\}$ such that $\Clr(R, B)(v) = 1$ if $v\in R$ and $-1$ if $v\in B$.
    When the coloring in the context is clear, we simply use $\Clr$.
    We will abuse the notation and treat $\Clr$ and $(R, B)$ as the same object.

    \item Given $G = (V, E)$, and a coloring $(R, B)$ of $V$, we denote the coloring on day $t$ by $(R_t, B_t)$ and denote $\Clr_t = \Clr(R_t, B_t)$.
    We will state clearly when we do not follow this convention.

    \item Given a coloring map $\Clr$, we define the reverse notations $R(\Clr) \defeq \Clr^{-1}(1)$ and $B(\Clr) \defeq \Clr^{-1}(-1)$.
    This makes $R_t = R(\Clr_t)$ and $B_t = B(\Clr_t)$.

    \item If a Red unanimity occurs, we say Red wins, and analogously for Blue.
\end{itemize}

\subsection{Main models and questions}  \label{sec:settings}  \label{sec:models}

The graph $G$ and the initial coloring $\Clr_0$ may be deterministically given or randomly generated from some joint distribution.
We use the following model:

\begin{itemize}
    \item $G$ is a random graph generated from the Erd\"{o}s-R\'{e}nyi $G(n, p)$ distribution: there are $n$ vertices, each pair of which is connected independently with probability $p$.
    The vast majority of works on majority dynamics use this graph model.

    \item $\Clr_0$ is chosen randomly and independently of $G$ by sampling $R_0$ uniformly among all subsets of $V$ with $n/2 + \adv$ elements.
    This is called the \textbf{fixed advantage coloring scheme} in \cite{TranVusparse}.
\end{itemize}

The numbers $n$, $p$ and $\adv$ are model parameters.
The independence between $G$ and $\Clr_0$ allows them to be generated in any order, simultaneously or not.
In fact, the discussion in \cite{TranVusparse} shows that the following models are equivalent (in terms of the joint distribution of $(G, \Clr_0)$)

\begin{enumerate}
    \item Fix $n$ vertices and their colors such that $|R_0| = n/2 + \adv$, then draw $G \sim G(n, p)$ over them.

    \item Fix $n$ vertices and their colors such that $|R_0| = |B_0| = n/2$, then draw $G \sim G(n, p)$ over them, then sample a set of ``defectors'' uniformly among subsets of $B_0$ of size $\adv$ and flip their colors. 
\end{enumerate}

Tran and Vu were among the first to consider the fixed advantage coloring scheme \cite{TranVuDense2020}.
They discovered the ``power of few'' phenomenon:
when $p$ is $\Theta(1)$, for any $\eps$, one can choose $\adv$ to be a sufficiently large constant to ensure Red wins with probability at least $1 - \eps$.
In other words, bribing a constant number of people is enough to win in a landslide, no matter how large the population is.
This constant is not large, as it is proven in \cite{TranVuDense2020} that for $p = 1/2$, $\adv = 6$ is enough for Red to win with probability $.99$.
Devlin and Berkowitz \cite{devlinberkowitz2022} later showed that $\adv = 3$ is enough for the winning odd to be at least $.6$.
Sah and Sawhney \cite{sahsawhney2021} later lowered this to $\adv = 1$, the smallest possible gap.

In real life, the average person knows only a tiny fraction of the population.
\emph{How many people need to be bribed for unanimous victory in this general case?}
Clearly, the density $p$ cannot be lower than the \emph{connectivity threshold} $n^{-1}\log n$, otherwise there will be, with high probability, small connected components with the same initial color.
These ``echo chambers'' will maintain their affiliation forever, no matter who the rest vote for.
Tran and Vu \cite{TranVusparse} formally posed the question

\begin{question}  \label{qn:optimal-power-of-few}
    Given $n$ and $p$ such that $(1 + \lambda)n^{-1}\log n \le p < 1$ for some constant $\lambda$, what is the smallest $\adv = \adv_{\min}(n, p, \eps)$ for which Red wins with probability at least $1 - \eps$?
\end{question}

It was shown in \cite{TranVusparse} that $\adv = \max\{ C_\eps p^{-1/2}, 10p^{-1} \}$ for a sufficient constant $C_\eps$ is enough, which interestingly does not depend directly on $n$.
The authors also conjectured \cite[Conjecture 1.11]{TranVusparse} that $\adv = C_\eps p^{-1/2}$ is the optimal ``power of few'' value.
This has been confirmed for the dense regime, where $p > \log^{-1/16}n$ in \cite{sahsawhney2021}.
The case for lower densities remains open.

Another popular coloring scheme, considered by most papers before \cite{TranVuDense2020}, is the \textbf{random 1/2 coloring scheme} \cite{TranVusparse}: $\{\Clr_0(v)\}_{v\in V}$ are i.i.d. random variables taking value in $\{\pm 1\}$ with equal probabilities.
(This is also equivalent to sampling $(R_0, B_0)$ uniformly among all possible partitions of $V$.)
Consider majority dynamics on $G\sim G(n, p)$ with $\Clr_0$ being the random 1/2 coloring.
The main question in this model is

\begin{question}  \label{qn:optimal-p-1/2-coloring}
    Given $n$ and $\eps$, what is the minimum $p = p_{\min}(n, \eps)$ for which the side with the initial majority wins with probability at least $1 - \eps$?
\end{question}

This question is related to Question \ref{qn:optimal-power-of-few} in the following way.
For any $\eps$, there is a constant $c_\eps$ such that an initial coloring from the random 1/2 scheme satisfies $|R_1| \ge (n + c_\eps \sqrt{n})/2$ with probability at least $1 - \eps/2$, by the central limit theorem.
Therefore, one can replace a random 1/2 coloring with a fixed coloring with $\adv = \Theta(\sqrt{n})$ to analyze any high-probability event.
Question \ref{qn:optimal-p-1/2-coloring} is then equivalent to:
\emph{What is the minimum $p$ such that $\adv_{\min}(n, p, \eps) = c_\eps\sqrt{n}$?}

If the optimal ``powef of few'' value is indeed $C_\eps p^{-1/2}$, it is tempting to deduce that the optimal $p_{\min}(n, \eps)$ is $(C_\eps/c_\eps)n^{-1}$.
However, the correct answer would be the connectivity threshold, more precisely $(1 + \lambda)n^{-1}\log n$ for any $\lambda > 0$, to avoid the ``echo chambers'' effect discussed above.

A very similar prediction was made in the pioneering paper by Benjamini, Chan, O'Donnell, Tamuz and Tan \cite{benjamini}.
They conjectured that $p = C_\eps n^{-1}$ is enough for \emph{$\eps$-unanimity:} the initially larger side having at least $(1 - \eps)n$ vertices in the final state.
It is likely that resolving Question \ref{qn:optimal-p-1/2-coloring} will also advance the progress towards Benjamini et al's conjecture.

\subsection{Current progresses and new results}

Fountoulakis, Kang and Makai \cite{fountoulakis} made the first progress towards resolving Question \ref{qn:optimal-p-1/2-coloring}.
They proved that the initial majority side wins in 4 days with probability at least $1 - \eps$ in the range $C_\eps n^{-1/2} \le p < 1$ for a sufficient constant $C_\eps$.
Interestingly, their proof first translates the context into the fixed advantage coloring scheme with $\adv = \Theta(\sqrt{n})$, without explicitly defining this scheme.

Tran and Vu independently defined the fixed advantage coloring scheme and discovered the ``power of few'' phenomenon for $p = \Theta(1)$ in \cite{TranVuDense2020}.
Their second paper \cite{TranVusparse} united progresses in both models, extending ``power of few'' for $(1 + \lambda)n^{-1}\log n \le p < 1$ (the best possible range) and reproducing the main result of \cite{fountoulakis} (with an extra day), and posed Question \ref{qn:optimal-power-of-few} explicitly.

In a recent breakthrough on Question \ref{qn:optimal-p-1/2-coloring}, Chakraborti, Kim, Lee and Tran \cite{jhk-etal2021} proved that the initial majority wins in 5 days with probability at least $1 - \eps$ if $p \ge C_\eps n^{-3/5}\log n$ for a constant $C_\eps$, and conjectured that their approach can be improved to work even for $p \ge C_\eps n^{-2/3}\log n$.

Our main result pushes the boundaries towards both questions and resolves the conjecture in \cite{jhk-etal2021}, even achieving a slightly better logarithmic factor.


\begin{theorem}  \label{thm:main-p<n^(-1/2)}
    Consider majority dynamics on a $G(n, p)$ random graph $G$ and a fixed initial coloring $\Clr_0$ with $|R_0| \ge n/2 + \adv$.
    For every $\eps$, there is a universal constant $C$ and constants $C_\eps$ and $N_\eps$ only depending on $\eps$ such that if $n \ge N_\eps$, $n^{-1/2}\log^{1/4}n \ge p \ge C n^{-1}\log^2 n$, and $\adv \ge C_\eps p^{-3/2}n^{-1/2}\log n$, Red wins within $O(\log_{pn}n)$ days with probability at least $1 - \eps$.
\end{theorem}

In other words, in the range $n^{-1/2}\log^{1/4}n \ge p \ge C n^{-1}\log^2 n$, the theorem above gives the answer $\adv = C_\eps p^{-3/2}n^{-1/2}\log n$ for Question \ref{qn:optimal-power-of-few}.
This is better than the answer $\max\{C'_\eps p^{-1/2}, 10p^{-1}\}$ in \cite{TranVusparse} for $p$ the above range, since
\begin{equation*}
    C_\eps p^{-3/2}n^{-1/2}\log n = \frac{C_\eps \log n}{p\sqrt{pn}} 
    \le \frac{C_\eps}{B_\eps^{1/2}p} \le \max\left\{\frac{10}{p}, \frac{C'_\eps}{p^{1/2}}\right\}
    ,\text{ given } C_\eps \text{ is large enough}.
\end{equation*}

Theorem \ref{thm:main-p<n^(-1/2)} also gives the answer $p = C_\eps n^{-2/3}\log^{2/3}n$ for Question \ref{qn:optimal-p-1/2-coloring}, improving on the answers $C_\eps n^{-3/5}\log n$ in \cite{jhk-etal2021} and $C_\eps n^{-1/2}$ in \cite{fountoulakis} and is the best known so far.
This is also the best (in terms of power of $n$) one can get from the approach in \cite{jhk-etal2021} (see the next section for details).

\begin{corollary}  \label{cor:p_log(n)n^(-2/3)}
Consider majority dynamics on a $G(n, p)$ random graph $G$ with a random $1/2$ initial coloring.
For each $\eps > 0$, there are is a constant $C_\eps > 0$ such that
if $p \ge C_\eps n^{-2/3}\log^{2/3} n$, then the initially larger side wins within $O(\log_{pn} n)$ days with probability at least $1 - \eps$.
\end{corollary}

\begin{proof}
Let $\mathcal{U}_R$ respectively be the event Red wins.
Define $\mathcal{U}_B$ analogously and let $\mathcal{U} = \mathcal{U}_R \vee \mathcal{U}_B$.
For each $\adv$, let $\mathcal{E}_\adv$ be the event that $|R_0| = n/2 + \adv$.
The following fact is due to, e.g. \cite[Lemma 3.1]{jhk-etal2021} and \cite[Lemma 3.1]{fountoulakis}:
there is a constant $c_\eps > 0$ such that
\begin{equation}  \label{eq:p_log(n)n^(-2/3)-temp1}
    \sum_{|\adv| \ge c_\eps\sqrt{n}} \Pr(\mathcal{E}_\adv)
    = \Pr\left(\max\{|R_0|, |B_0|\} \ge \frac{n}{2} + c_\eps\sqrt{n}\right)
    \ge 1 - \frac{\eps}{3}.
\end{equation}
We also have
\begin{equation} \label{eq:p_log(n)n^(-2/3)-temp2}
\begin{aligned}
    \Pr(\mathcal{U})
    & \ge \sum_{\adv \ge c_\eps\sqrt{n}} \Pr(\mathcal{U}\mid \mathcal{E}_\adv)\Pr(\mathcal{E}_\adv)
    + \sum_{\adv \ge c_\eps\sqrt{n}} \Pr(\mathcal{U}\mid \mathcal{E}_{-\adv})\Pr(\mathcal{E}_{-\adv}) - \eps/3.
\end{aligned}
\end{equation}
Take the constants $C$ and $C_\eps$ from Theorem \ref{thm:main-p<n^(-1/2)}.
Suppose $p \ge A_\eps n^{-2/3}\log^{2/3}n$ for a constant $A_\eps$ to be chosen later.
Then clearly $p \gg n^{-1}\log^2 n$.
For each $\adv \ge c_\eps\sqrt{n}$, we have
$\adv p^{3/2}n^{1/2} \ge c_\eps A_\eps^{3/2}\log n$,
Choose $A_\eps$ so that $c_\eps A_\eps^{3/2} > C_\eps$,
we can then apply Theorem \ref{thm:main-p<n^(-1/2)} to get, for a constant $C_1$,
\begin{equation*}
	\Pr(\mathcal{U} \mid \mathcal{E}_\adv) \ge \Pr(\mathcal{U}_R \mid \mathcal{E}_\adv) \ge 1 - \frac{C_1\log^2 n}{\adv^2p^3n}
	\ge 1 - \frac{C_1}{c_\eps^2A_\eps^3}.
\end{equation*}
Now choose $A_\eps$ so that $c_\eps^2 A_\eps^3 > 3C_1/\eps$, the right-hand side is at least $1 - \eps/3$.
Then by symmetry, $\Pr(\mathcal{U} \mid \mathcal{E}_{-\adv}) \ge 1 - \eps/3$.
Therefore, Eq. \eqref{eq:p_log(n)n^(-2/3)-temp2} becomes
\begin{equation*}
\begin{aligned}
    \Pr(\mathcal{U})
	&
    \ge \sum_{|\adv| \ge c_\eps\sqrt{n}} \left(1 - \frac{\eps}{3}\right) \Pr(\mathcal{E}_\adv) - \frac{\eps}{3}
	= \left(1 - \frac{\eps}{3}\right) \sum_{|\adv| \ge c_\eps\sqrt{n}} \Pr(\mathcal{E}_\adv) - \frac{\eps}{3}
    \ge \left(1 - \frac{\eps}{3}\right)^2 - \frac{\eps}{3} \ge 1 - \eps,
\end{aligned}
\end{equation*}
where the second inequality is due to Eq. \eqref{eq:p_log(n)n^(-2/3)-temp1}.
The proof of Corollary \ref{cor:p_log(n)n^(-2/3)} is complete.
\end{proof}

Next, we comment on the theoretically optimal answers to Questions \ref{qn:optimal-power-of-few} and \ref{qn:optimal-p-1/2-coloring}.
We will show that the results by Tran and Vu \cite{TranVusparse} and Fountoulakis et al \cite{fountoulakis} are the best possible for any ``first-day'' approach,
while Theorem \ref{thm:main-p<n^(-1/2)} and Corollary \ref{cor:p_log(n)n^(-2/3)} are the best for any ``second-day'' approach.
The exact meanings of ``first-'' and ``second-day'' will be clarified in the next section.

\subsection{Theoretical limits for ``power of few''}  \label{sec:optimality}


For convenience, let $\adv_t \defeq \frac{1}{2}(|R_t| - |B_t|)$.
In their seminal paper \cite{benjamini}, Benjamini et al predicted that for $p \ge (1 + \lambda)n^{-1}\log n$, the dynamics progresses in two phases (assuming $\adv_0$ is large enough):
\begin{enumerate}
    \item \emph{Slow expansion:} The Red advantage grows by a $\Theta(\sqrt{pn})$ factor every day, until it is more than $2\eps n$, for some small constant $\eps$ (meaning Red has more than $\left(\frac{1}{2} + \eps\right)n$ members).

    \item \emph{Landslide:} Once the threshold $\left(\frac{1}{2} + \eps\right)n$ is passed, the Red camp rapidly converts everyone.
    Blue will shrink by a $\Theta(pn)$ factor every day, leading to a Red win after $O(\log_{pn}n)$ days.
\end{enumerate}
All papers in majority dynamics that we know of so far follow this roadmap.

The landslide phase has been proven to happen regardless of the members in each side, independently in \cite{fountoulakis} and \cite{TranVusparse}, the latter for the whole range $p \ge (1 + \lambda)n^{-1}\log n$.

In fact, it is shown in \cite{TranVusparse} that once $\adv_k \ge C\sqrt{\frac{n}{p}}$ for some sufficient constant $C$ (the day before the landslide), the graph is so well-connected that Red wins no matter who its members are.

The slow expansion phase is tricky because the advantage has not reached a point where one does not need to care about the members of each camp.

Most arguments so far deal with this phase as a whole, meaning they choose some $k$ and prove directly that $\adv_k \ge  C\sqrt{\frac{n}{p}}$, rather than showing $\adv_t \ge c\sqrt{pn}\adv_{t - 1}$ for all $1\le t\le k$.
The latter is very hard to prove since for $t \ge 2$, one does not have the independence between the Red and Blue neighborhoods of a vertex $v$ to analyze $\Clr_t(v)$ anymore.

The approaches by Tran and Vu \cite{TranVuDense2020,TranVusparse}, Fountoulakis et al \cite{fountoulakis}, and Devlin and Berkowitz \cite{devlinberkowitz2022} and Sah and Sawhney \cite{sahsawhney2021} are \textbf{first-day approaches}, since they avoid this loss of independence by proving $\adv_1 \ge C\sqrt{\frac{n}{p}}$ under the right conditions.
One needs $\adv_0 = \adv$ to be large enough for this to hold, and since $\adv_1 = \Theta(\adv_0\sqrt{pn})$, one needs $\adv = \Omega(Cp^{-1})$ for this aproach to work.

The approach by Chakraborti et al \cite{jhk-etal2021} and this paper is a \textbf{second-day approach}.
We do not make any statement about $\adv_1$ and prove directly that $\adv_2 \ge C\sqrt{\frac{n}{p}}$.
Since $\adv_2 = \Theta(\adv_0pn)$, the best one can get out of this approach is $\adv = \Omega(p^{-3/2}n^{-1/2})$.
This means Theorem \ref{thm:main-p<n^(-1/2)} is optimal for any second-day approach, except for the log factor.

\emph{What does an optimal value for $\adv$ look like?}
To kick off the first phase, we need $\adv_0$ to not be too small.
The analysis in \cite{TranVusparse} shows that $\adv_1$ (as a function of the random graph $G$) has expectation $\Theta(\adv\sqrt{pn})$ and standard deviation $\Theta(\sqrt{n})$.
Informally, we need $\E{\adv_1} \ge C_\eps\sqrt{\Var{\adv_1}}$ for a sufficiently large $C_\eps$ to assert that $\adv_1 = \Omega(\E{\adv_1})$ with probability at least $1 - \eps$.
This implies $\adv \ge C_\eps p^{-1/2}$ is the lowest possible advantage.
This has been reached by Sah and Sawhney \cite{sahsawhney2021} in their breakthrough paper for the dense regime.
To extend it to the lower densities, one likely needs to find a general \textbf{k-day approach} to deal with the slow expansion phase.

When translating to the model with the random 1/2 coloring scheme to find $p_{\min}(n, \eps)$, it is not hard to see that the results of Tran and Vu \cite{TranVusparse} and Fountoulakis et al \cite{fountoulakis}, namely $p = C_\eps n^{-1/2}$ are optimal for any first-day approach;
while Corollary \ref{cor:p_log(n)n^(-2/3)}, which gives $p = C_\eps n^{-2/3}\log^{2/3}n$, is optimal up to the log factor for any second-day approach.

The rest of the paper is organized as follows:
In section \ref{sec:main-proof}, we prove Theorem \ref{thm:main-p<n^(-1/2)} after introducing two technical lemmas.
These lemmas will be proven in section \ref{sec:technical-lemmas}.
Section \ref{sec:binom-bound} is reserved for utility estimates on binomial random variables used throughout the proofs.

\section{Proof of main results}  \label{sec:main-proof}

\subsection{Notation and proof sketch}

For the proof, we will add many new notations.
The ones currently in use are also reminded here.
\begin{itemize}
    \item For $u, v\in V$, let $u\adj v$ denote $u$ and $v$ being adjacent.
    
    
    \item For each $v\in V$, let $N(v)$ be the set of $v$'s neighbors and $d(v) \defeq |N(v)|$ be its degree.
    For each subset $S\subset V$, let $d_S(v) \defeq |N(v)\cap S|$.
    
    
    \item Let $a_1 \vee a_2 \vee \ldots \vee a_n \defeq \max\{a_i\}_{i=1}^n$ and $a_1 \wedge a_2 \wedge \ldots \wedge a_n \defeq \min\{a_i\}_{i=1}^n$.
    
    \item When three terms $X, Y, E$ satisfy $|X - Y|\le E$, we write $X = Y \pm E$.
    While this notation does not obey standard arithmetic rules, we only need the following trivial properties:
    \begin{enumerate}
        \item For all $X, Y$ and all $E \ge 0$, $X = Y \pm E \iff X - Y = \pm E \iff Y = X \pm E$.
        \item For all $X, Y, Z$ and all $E_1, E_2\ge 0$, $X = Y \pm E_1 \wedge Y = Z \pm E_2 \implies X = Z \pm (E_1 + E_2)$.
    \end{enumerate}
    When $X$, $Y$ and $E$ are functions of $n$, and $|X - Y| = O(E)$, we write $X = Y \pm O(E)$.
    
    \item Let $\Phi$ be the standard Gaussian CDF, i.e. $\Phi(x) \defeq (2\pi)^{-1/2}\int_{-\infty}^x e^{-t^2/2}dt$.
    
    \item When $x(n), y(n)$ are positive functions and $\liminf_{n\to \infty} \frac{x(n)}{y(n)} = C$ for a constant $C > 0$, we write $x(n) \gtrsim y(n)$.
    If $\limsup_{n\to \infty} \frac{x(n)}{y(n)} = 0$, we write $x(n) \ll y(n)$.

    \item For a coloring $(R, B)$ on $V$, the corresponding coloring map is $\Clr = \Clr(R, B): V\to \{\pm 1\}$ such that $\Clr(v) = 1$ for $v\in R$ and $-1$ for $v\in B$.
    Conversely, we write $R = R(\Clr)$ and $B = B(\Clr)$.

    \item For subsequent days, we let $R_t = R(\Clr_t)$ and $B_t = B(\Clr_t)$ and in turn $\Clr_t = \Clr(R_t, B_t)$.
\end{itemize}

The proof of Theorem \ref{thm:main-p<n^(-1/2)} follows closely the structure of the argument in \cite{jhk-etal2021}.
Our main innovations include repeating the computations in a more general setting by replacing the initial gap of $c\sqrt{n}$ in their paper with the parameter $\adv$, and improving \cite[Lemma 1.3]{jhk-etal2021} by handling some expressions more delicately with precise cancellations.


For the sake of completeness, we will present the entirety of the new proof, except where results from \cite{jhk-etal2021} and other works are used explicitly.
The proof follows two steps, as discussed in Section \ref{sec:optimality}

\begin{itemize}
    \item \emph{Slow expansion.} We show that with high probability $|R_2| \ge \frac{n}{2} + C\adv pn$ for a constant $C$.

    \item \emph{Landslide.} We use results from \cite{TranVusparse,fountoulakis} to show that Red wins within $O(\log_{pn}n)$ days from the third day.
\end{itemize}

Per the discussion in Section \ref{sec:optimality}, the bulk of the technical challenge comes from the first step.
The next section will outline the intuition behind our approach, including a crucial perspective change that makes the second-day approach work.

\subsection{First Two Days}

Below is the main result of the first two days:
\begin{lemma}  \label{lem:day2}
    Consider a majority dynamics process on a random graph $G\sim G(n, p)$ with $p \le cn^{-1/2}$ for an arbitrary constant $c$, and an initial coloring $\Clr_0$ with $\lceil n/2 + \adv \rceil$ Red vertices.
    There are universal constants $C$ (large enough) and $c > 0$ (small enough) so that the following holds:
    For each $\eps > 0$, there is a large enough constant $C_\eps$ such that if
    \begin{equation*}
        pn \ge C\log^2 n
        \quad \text{ and } \
        \adv p^{3/2}n^{1/2} \ge C_\eps\log n,
    \end{equation*}
    then with probability $1 - \eps$, we have $\sum_{v\in V} \Clr_2(v) \ge 2c \adv pn$, meaning $|R(\Clr_2)| \ge n/2 + c\adv pn$.
\end{lemma}

The approach in \cite{jhk-etal2021} is to first consider a \emph{balanced coloring} $\bClr_0$, where each side has $n/2$ vertices, then draw the graph, then uniformly sample a set $S$ of $\adv$ ``defectors'' from $B(\bClr_0)$ and flip them to form $\Clr_0$, meaning $R(\Clr_0) = R(\bClr_0) \cup S$.

If one compares $\bClr_1$ and $\Clr_1$, some vertices in will have their day 1 color switch from Blue to Red.
If enough such switching happens in $N(v)$, the odds of $v$ turning Red increases significantly.

This motivates the following definition:
\begin{definition}  \label{defn:almost-red}
    Given $D > 0$, a vertex $v$ is $D$-almost Red if $\sum_{u\in N(v)}\bClr_1(u) \ge -D$.
\end{definition}
This generalizes the notion of \emph{almost positive} vertices in \cite{jhk-etal2021}.

In the proof, we will choose $D = \Theta(\adv p^{3/2}n^{1/2})$ with an appropriately chosen linear factor.
This ensures that with high probability, after $\adv$ random Blue vertices defect to form $\Clr_0$, each of the $D$-almost Red vertices has more than $D$ flipping neighbors, making them Red in day 2.

We formalize this idea in the following two lemmas.
The first states that from a balanced initial coloring, there will be $\frac{n}{2} + C\sqrt{\frac{n}{p}}$ $D$-almost Red vertices, waiting to flip on day 2 after the defection.

\begin{lemma}[Stronger Lemma 1.3 in \cite{jhk-etal2021}] \label{lem:almost-pos-v}
Consider majority dynamics on a random graph $G\sim G(n, p)$ with $p \le n^{-1/2}\log^{1/4}n$ and an initial coloring $\bClr_0$ with $\lceil n/2 \rceil$ Red vertices.
There are universal constants $C_1$, $C_2$ such that the following holds:

If $pn \ge C_1\log^2 n$, then for every constant $B > 0$ and $D$ such that $B\sqrt{pn} \ge D \ge C_2\log n$, there is another constant $c_B > 0$ depending only on $B$ such that the number of $D$-almost Red vertices in $G$ is at least $\frac{n}{2} +  c_BD\sqrt{\frac{n}{p}}$ with probability at least $1 - O\left(\frac{\log^2 n}{c_B^2D^2}\right)$.
\end{lemma}



The second lemma states that after the initial defection, every vertex $v$ will have many neighbors that defect after day 1.
If $v$ also happens to be $D$-almost Red, for a small enough $D$, $v$ will defect after day 2.

\begin{lemma}  \label{lem:unst-neighs-with-swing-neighs}
    Consider majority dynamics on $G\sim G(n, p)$ with an initial coloring $\bClr_0$ with $\lceil n/2 \rceil$ Red vertices.
    Suppose that $\adv p^{3/2}n^{1/2} \ge 16\log n$ and $p\adv \le B$ for some constant $B > 0$. 
    Suppose a set $S$ of defectors is sampled uniformly among subsets of $B(\bClr_0)$ with size $\adv$ to form the coloring $\Clr_0$.
    Then there is a sufficiently small constant $c_B > 0$ such that, with probability $1 - O(n^{-1})$, every $v\in V$ has at least $c_B \adv p^{3/2}n^{1/2}$ neighbors $u$ such that $\bClr_1(u) = -1$ but $\Clr_1(u) = 1$.
\end{lemma}

The proof of these lemmas involve heavy calculations that use many precise bounds for the binomial distribution, so we put them in Section \ref{sec:technical-lemmas}.

Assume for now that these two lemmas hold, we can prove Lemma \ref{lem:day2}.

\begin{proof}[Proof of Lemma \ref{lem:day2}]
    Fix $\eps > 0$.
    Suppose $pn \ge C_1\log^2 n$ for the $C_1$ from Lemma \ref{lem:almost-pos-v} and $\adv p^{3/2}n^{1/2} \ge C_\eps\log n$ for some $C_\eps$ depending on $\eps$.
    We will show that one can choose $C_\eps$ large enough to have $|R_2| \ge \frac{n}{2} + 3\sqrt{\frac{n}{p}}$ with probability at least $1 - \eps$.


    Suppose $\adv p < 10$ for now.
    Let $c_1$ be $c_{10}$ from Lemma \ref{lem:unst-neighs-with-swing-neighs} and let $D = c_1\adv p^{3/2}n^{1/2}$.
    Then with probability $1 - O(n^{-1})$, everyone has at least $D$ neighbors that flip after day 1.

    We condition $G$ and $S$ on this event.
    Now every $D$-almost Red vertices will flip after day 2, by Definition \ref{defn:almost-red}.
    We want to apply Lemma \ref{lem:almost-pos-v}.
    We have $10c_1\sqrt{pn} \ge D \ge c_1C_\eps \log n$ since $p\adv \le 10$.
    Let $B = 10c_1$ and $c_2 = c_B$ from Lemma \ref{lem:almost-pos-v}.
    Take $C_2$ from this lemma, we ensure that $c_1C_\eps \ge C_2$ to apply it.
    We will get at least $\frac{n}{2} + c_2D\sqrt{\frac{n}{p}} = \frac{n}{2} + c_1c_2\adv pn$ $D$-almost Red vertices, with probability at least $1 - \frac{A\log^2 n}{c_2^2D^2}$ for some universal constant $A$.
    
    For the fail probability to be not more than $\eps$, we want $C_\eps \ge c_2^{-1}\sqrt{A/\eps}$.
    So we can choose
    \begin{equation*}
        C_\eps \defeq 2c_2^{-1}\sqrt{A/\eps} \vee c_1^{-1}C_2
    \end{equation*}
    to satisfy the conditions of both lemmas.
    
    Now to be able to choose $D$ at all, we need $10c_1\sqrt{pn} \ge C_\eps\log n$.
    However, this is only needed if we follow the second-day approach so far.
    If $10c_1\sqrt{pn} < C_\eps\log n$, this means $p\adv > 10$ so the first-day approach in \cite{TranVusparse} implies Red wins with probability $1 - o(1)$.
    The proof is complete.
\end{proof}

\subsection{Third Day and beyond}

Note that when $\adv pn \gtrsim \log n\sqrt{n/p}$ and $pn \gtrsim \log^2n \gg \log n$, we can apply either \cite[Lemma 2.4]{TranVusparse} then \cite[Lemma 2.5]{TranVusparse}, or \cite[Lemma 4.2]{jhk-etal2021}, to arrive at the inequalities
\begin{equation}  \label{eq:main-temp1}
    |B_3| \le .45n
    \ \text{ and } \
    |B_{t+1}| \le \frac{C}{pn}|B_t|,
    \ \text{ for each } t \ge 3,
    \ \text{ until } |B_t| \le \frac{pn}{4},
\end{equation}
with probability at least $1 - \eps/2$.
When there are at most $pn/4$ Blue vertices left, Red will win the next day if we also condition that every vertex has at least $pn/2$ neighbors, which happens with probability $1 - o(1)$.
The proof of Theorem \ref{thm:main-p<n^(-1/2)} is complete.

\section{Proof of technical lemmas}  \label{sec:technical-lemmas}

Before beginning, we introduce some extra notation to aid the presentation.

\begin{itemize}
    \item Let $S$ denote the set of defectors, or ``swing'' set.
    
    \item Similar to the notations for $\Clr$, let $\bR_t \defeq R(\bClr_t)$ and $\bB_t \defeq B(\bClr_t)$.
    Conversely $\bClr_t = \Clr(\bR_t, \bB_t)$.
    Note that $S\subset \bB_0$ and $R_0 = \bR_0 \cup S$.

    \item For each $v\in V$ and $t\in \N$, let $\dif_t(v) \defeq d_{R_t}(v) - d_{B_t}(v)$ and $\bdif_t(v) \defeq d_{\bR_t}(v) - d_{\bB_t}(v)$.

    \item For each subset $U\subset V$, let $G(U)$ be the induced subgraph over $U$.
\end{itemize}

\subsection{Proof of Lemma \ref{lem:unst-neighs-with-swing-neighs}}

To prove Lemma \ref{lem:unst-neighs-with-swing-neighs}, we follow the argument in the first half of the proof of \cite[Lemma 3.3]{jhk-etal2021}.

\begin{proof}[Proof of Lemma \ref{lem:unst-neighs-with-swing-neighs}]
    Let $D = c\adv p^{3/2}n^{1/2}$ for a constant $c$ chosen later.
    Fix an arbitrary choice of the $S$.
    This also fixes the coloring $\Clr$.
    Temporarily call the vertices $u$ such that $\bClr_1(u) = -1$ and $\Clr_1(u) = 1$ \emph{flipping}.
    It suffices to show that with probability $1 - O(n^{-1})$, every $v\in V$ has at least $D$ flipping neighbors, since the choices of $S$ are symmetrical.

    Let us narrow down to a special type of flipping neighbors that we can control.
    We call $u\in N(v)$ \emph{vulnerable} if $d_{\bR_0}(u) = d_{\bB_0\setminus S}(u)$ and $u$ has at least 1 neighbor in $S$.
    If $u$ is vulnerable, we have
    \begin{equation*}
        d_{\bR_0}(u) - d_{\bB_0}(u)
        = d_{\bR_0}(u) - d_{\bB_0\setminus S}(u) - d_S(u)
        = - d_S(u) < 0
        \implies \bClr_1(u) = -1.
    \end{equation*}
    We also have
    \begin{equation*}
        d_{R_0}(u) - d_{B_0}(u)
        = d_{\bR_0}(u) + d_S(u) - (d_{\bB_0}(u) - d_S(u))
        = d_S(u) < 0
        \implies \Clr_1(u) = 1.
    \end{equation*}
    Therefore a vulnerable vertex is flipping.
    The advantage is that we can easily lower bound the number of vulnerable neighbors of $v$.
    We claim that $v$ has at least $D$ vulnerable neighbors with probability $1 - O(n^{-2})$.

    We claim that with high probability $|N(v)\cap S| < .1\adv$.
    This is a $\Bin(\adv, p)$ random variable with mean $p\adv < 10$, so it is close to a Poisson variable.
    By Lemma \ref{lem:poisson-bound},
    \begin{equation*}
        \Pr\left(|N(v)\cap S| \ge .1\adv\right)
        \le 2\left(\frac{e\adv p}{.1\adv} \right)^{.1\adv}
        \le 2(30p)^{.1\adv} < n^{-2},
    \end{equation*}
    where the last inequality holds if $\adv > 20\log n$, which holds if $\adv p^{3/2}n^{1/2} \ge \log n$ and $p < n^{-1/2}\log n$.

    For each $U\subset V\setminus \{v\}$, we say $U$ is \emph{regular} if $|U\cap S| < .1\adv$ and $|U| \ge pn/2$.
    By the above and a Chernoff bound, $N(v)$ is regular with probability $1 - O(n^{-2})$.

    Fix a regular $U\subset V\setminus \{v\}$ and condition the probability space of $G$ on the event $U = N(v)$.
    Let $V^+ \defeq \bR_0 \setminus \{v\}$ and $V^- \defeq \bB_0 \setminus \{v\}$.
    Let $W \defeq (V\setminus \{v\}) \setminus U$ and let $W^+ \defeq V^+ \cap W$ and $W^- \defeq V^- \cap W$.
    We have $|W^-\cap S| = |W\cap S| = |S| - |U\cap S| - \one_S(v) > .8\adv$.

    
    Let $I_u$ be the event that $u$ is vulnerable.
    It is the conjunction of the two independent events: $\{d_{\bR_0}(u) = d_{\bB_0\setminus S}(u)\}$ and $\{d_S(u) \ge 1\}$.
    Consider the former, we have
    \begin{equation*}
        \Pr\left(
            d_{\bR_0}(u) = d_{\bB_0\setminus S}(u)
        \right)
        = \Pr\left(
            d_{V^+}(u) + \Clr_0(v) = d_{V^-\setminus S}(u)
        \right)
        \ge \frac{r_B}{\sqrt{p(1 - p)|V\setminus S|}}
        \ge \frac{r_B}{\sqrt{pn}},
    \end{equation*}
    for a constant $r_B$ depending on $B$.
    For the latter, we have
    \begin{equation*}
        \Pr(d_S(u) \ge 1) \ge \Pr(d_{W^-\cap S}(u) \ge 1) \ge 1 - (1 - p)^{|W^-\cap S|}
        \ge 1 - (1 - p)^{.8\adv}
        \ge 1 - e^{-.8p\adv} \ge s_Bp\adv,
    \end{equation*}
    for a constant $s_B$ depending on $B$, where the last inequality is due to $p\adv \le B$.
    Therefore
    \begin{equation*}
        \Pr(I_u) = \Pr\left(
            d_{\bR_0}(u) = d_{\bB_0\setminus S}(u)
        \right) \cdot \Pr(d_S(u) \ge 1)
        \ge \frac{r_Bs_B p\adv}{\sqrt{pn}}.
    \end{equation*}
    Since $\{I_u\}_{u\in U}$ are independent after exposing $U$ (note that $S$ is fixed), and there are $|U| \ge pn/2$ of them, the number of vulnerable neighbors of $v$ dominates $X \sim \Bin\bigl(\frac{pn}{2}, r_Bs_B\adv \sqrt{\frac{p}{n}}\bigr)$ in probability.
    Note that $\E{X} = \frac{1}{2}r_Bs_B \adv p^{3/2}n^{1/2} = \frac{r_Bs_B}{2c}D$.
    Choose $c = c_B = r_Bs_B/4$,
    by a Chernoff bound,
    \begin{equation*}
        \Pr(X < D)
        = \Pr\left(X < \frac{1}{2} \E{X} \right)
        \le \exp\left[
            \frac{1}{8}\adv p^{3/2}n^{1/2}
        \right]
        \le n^{-2},
    \end{equation*}
    given $\adv p^{3/2}n^{1/2} \ge 16\log n$.
    Since this holds for all regular $U$ and the probability of an irregular $U$ is less than $n^{-2}$, the total exceptional probability is $O(n^{-2})$.
    The proof is complete.
\end{proof}

\subsection{Proof of Lemma \ref{lem:almost-pos-v}}

Let us move on to the proof of Lemma \ref{lem:almost-pos-v}.
Since $D$ is fixed throughout the proof, we simply let $A_{v}$ denote the event that $v$ is $D$-almost Red.
We abuse the notation and let $A$ be the set of $D$-almost Red vertices.
We have $|A| = \sum_{v\in V}\one_{A_v}$, so
\begin{equation*}
    \E{|A|} = \sum_{v\in V} \Pr(A_v)
    \ \text{ and } \
    \Var{|A|} = \sum_{v\in V} \Var{\one_{A_v}}
    + \sum_{v_1\neq v_2\in V} \Cov{\one_{A_{v_1}}, \one_{A_{v_2}}}
\end{equation*}
Our strategy is simple: we lower bound $\E{|A|}$ by lower bounding each $\Pr(A_v)$, then upper bound $\Var{|A|}$ by upper bounding $\Cov{\one_{A_{v_1}}, \one_{A_{v_2}}}$ for each distinct pair $(v_1, v_2)$ (note that $\Var{\one_{A_v}} \le 1/4$ for each $v$).
Once the two bounds are available, we can just use Chebyshev's inequality \cite{chebyshev} to lower bound $|A|$ with high probability.

\subsubsection{Step 1: Lower bounding $\E{|A|}$}  \label{sec:almost-pos-set-exp}

Here is the main result of this step.
Throughout this section and the next, we assume the objects and conditions of Theorem \ref{thm:main-p<n^(-1/2)}.

\begin{claim}  \label{claim:almost-pos-set-exp}
    Let $A$ be the set of $D$-almost positive vertices.
    Let $B > 0$ be an arbitrary constant
    There are universal constants $C_1$ and $C_2$ (large enough) and a constant $c_B$ depending only on $B$ (small enough) such that if $pn \ge C_1\log^2 n$ and $B\sqrt{pn} \ge D \ge C_2\log n$, then
    \begin{equation*}
        \E{|A|} \ge \frac{n}{2} + c_BD\sqrt{\frac{n}{p}}.
    \end{equation*}
\end{claim}

To prove this claim, we want to lower bound $\Pr(A_v)$ for each $v\in V$.
We have the following claim
\begin{claim}  \label{claim:P(A_v)}
    Consider the objects and assumptions of Lemma \ref{lem:almost-pos-v}.
    Let $Z^+, Z^-$ be two independent copies of $\Bin(n/2 - \lceil pn \rceil, p)$.
    We have
    \begin{equation}  \label{eq:P(A_v)}
        \E{\one_{A_v}} \ge \Phi\left(\frac{D}{2\sqrt{pn}} + 2\bClr_0(v)\Pr(Z^+ = Z^-)\sqrt{pn}\right)
		\pm O\left(\frac{1}{\sqrt{pn}}\right).
    \end{equation}
\end{claim}

Suppose this holds, then we can prove Claim \ref{claim:almost-pos-set-exp}.

\begin{proof}[Proof of Claim \ref{claim:almost-pos-set-exp}]
    Claim \ref{claim:P(A_v)} implies for $|A| = \sum_{v\in V}\one_{A_v}$:
    \begin{equation*}
        \E{|A|} \ge \sum_{v\in V}
        \Phi\left(x_v
        \right) \pm O\left(\frac{\sqrt{n}}{\sqrt{p}}\right),
        \ \text{ where } \ 
        x_v \defeq \frac{D}{2\sqrt{pn}}
        + 2\bClr_0(v)\Pr(Z^+ = Z^-)\sqrt{pn}.
    \end{equation*}
    Note that each term inside $\Phi(\cdot)$ is less than $20$ in absolute value.
    Since the number of vertices $v$ for which $\bClr_0(v)$ are $1$ and $-1$ are both $n/2$, we can form pairs of opposite colors.
    For each pair $v, v'$, we have, by \cite[Lemma 2.4]{jhk-etal2021}, for a fixed positive function $f_0$,
    \begin{equation*}
        \Phi(x_v) + \Phi(x_{v'})
        \ge 1 + f_0(20)(x_v + x_{v'})
        = 1 + \frac{f_0(20)D}{\sqrt{pn}}.
    \end{equation*}
    This implies, if one allows $D \ge \log n$,
    \begin{equation*}
        \E{|A|} \ge \frac{n}{2} + \frac{f_0(20)D\sqrt{n}}{\sqrt{p}}
        \pm O\left(\frac{\sqrt{n}}{\sqrt{p}}\right)
        \ge \frac{n}{2} + \frac{f_0(20)D}{2}\sqrt{\frac{n}{p}}.
    \end{equation*}
    The proof is complete.
\end{proof}

The main technical challenge is Claim \ref{claim:P(A_v)}.
Fix $v\in V$.
We have $\Pr(A_v) = \Pr(\bdif_1(v) \ge -D)$.
There are two ways to write $\bdif_1(v)$ as a sum of random variables.
The first choice is $\bdif_1(v) = \sum_{u\in N(v)} \bClr_1(u)$, which is hard to analyze because $N(v)$ is random.
The second is $\bdif_1(v) = \sum_{u\in V} \bClr_1(u)\one_{u\adj v}$.
This is still hard because $\{\bClr_1(u)\}_{u\in V}$ are dependent random variables.

The key actually lies in the first method.
Fix $U\subset V$ and condition $G$ on the event $U = N(v)$, then the first sum has a fixed number of summands.
Further a subgraph $H$ over $U$ and condition on $G(U) = H$, the variables in the sum are now independent.

The vast majority of cases for $U$ will be \emph{regular}, meaning the statistics of $U$ and $G(U)$ are close to their expected values.
For convenience, we define the notations (similar to Lemma \ref{lem:unst-neighs-with-swing-neighs}'s proof):
\begin{itemize}
    \item $U^+ = \bR_0 \cap U$ and $U^- = \bB_0 \cap U$.

    \item $W = (V\setminus \{v\}) \setminus U$, then $W^+ = \bR_0 \cap W$ and $W^- = \bB_0 \cap W$.

    \item $V^+ = U^+ \cup W^+$ ($= \bR_0 \setminus \{v\}$ and $V^- = U^- \cup W^- = \bB_0 \setminus \{v\}$.

    \item We call the set $U\subset V$ and the subgraph $H$ drawn over $U$ \emph{regular} if they satisfy all conditions below.
    \begin{equation*}  \label{eq:regular-subset}
    \begin{aligned}
        |U^+| = \frac{pn}{2} \pm 3\sqrt{pn\log n},
        \ \quad \
        |U^-| = \frac{pn}{2} \pm 3\sqrt{pn\log n},
        \\
        \|\mathbf{d}_U\|_\infty \le \log n,
        \quad
        \left| \one_{|U|}^T\mathbf{d}_U \right| \le 5\log n,
        \ \text{ and } \
        \|\mathbf{d}_U\|_2 \le 3\sqrt{pn\log n},
    \end{aligned}
    \end{equation*}
    where $\mathbf{d}_U \defeq [d_{U^+}(u, H) - d_{U^-}(u, H)]_{u\in U}$ is a vector in $\R^{|U|}$.
\end{itemize}

Simple Chernoff bounds show that $U$ is regular with probability $1 - O(n^{-2})$.
The strategy is to analyze $\bdif_1(v)$ conditioned on a regular $U = N(v)$ and $H = G(U)$ to take advantage the independence of its summands.
Let us proceed with the formal proof.

\begin{proof}[Proof of Claim \ref{claim:P(A_v)}]
    Fix subset $U\subset V$ and subgraph $H$ over $U$, that are both regular, and condition $G$ on the events $N(v) = U$ and $H = G(U)$.
    Until noted otherwise, we omit them from the probability notation.
    
    As discussed above, we write
    $\Pr(A_v) = \Pr(\bdif_1(v) \ge -D)$, where $\bdif_1(v) = \sum_{u\in U} \bClr_1(u)$ is a sum of independent random variables.
    By the Berry-Esseen theorem \cite{esseen}, one can replace this sum with a Gaussian variable with the same expectation and variance, at the cost of some error.
    We have
    \begin{equation*}
        \E{\bdif_1(v)} = \sum_{u\in U} \E{\bClr_1(u)} = \sum_{u\in U} \eps_u
        \ \text{ where } \
        \eps_u \defeq \E{\bClr_1(u)}.
    \end{equation*}
    The notation $\eps_u$ is convenient because one can easily write
    \begin{equation*}
        \Var{\bClr_1(u)} = 1 - \eps_u^2
        \ \text{ and } \ 
        \E{\left|\bClr_1(u) - \E{\bClr_1(u)}\right|^3} = 1 - \eps_u^4,
    \end{equation*}
    the latter being important for the error term in the Berry-Esseen theorem.
    By the Berry-Esseen theorem, we have
    \begin{equation}  \label{eq:P(A_v)-proof-esseen1}
    \begin{aligned}
        \Pr(A_v) 
        & = \Phi\left(
            \frac{
                D + \E{\bdif_1(v)}
            }{
                \left(\Var{\bdif_1(v)}\right)^{1/2}
            }
        \right)
        \pm C_0 \frac{
            \sum_{u\in U} \E{\left|\bClr_1(u) - \E{\bClr_1(u)}\right|^3}
        }{
            \left(\Var{\bdif_1(v)}\right)^{1/2}
        }
        \\
        & = \Phi\left(
            \frac{
                D + \sum_{u\in U} \eps_u
            }{
                \left(|U| - \sum_{u\in U} \eps_u^2\right)^{3/2}
            }
        \right)
        \pm C_0 \frac{
            |U| - \sum_{u\in U} \eps_u^4
        }{
            \left(|U| - \sum_{u\in U} \eps_u^2\right)^{3/2}
        }.
    \end{aligned}
    \end{equation}
    Let us compute $\eps_u$ for $u\in U$ and then move on to their sum.
    Ideally, we want $\eps_u$ to be small, meaning $\Pr(\bClr_1(u) = 1)$ to be close to $1/2$.
    We have
    \begin{equation*}
    \begin{aligned}
        & \Pr(\bClr_1(u) = 1) = \Pr\left(
            d_{V^+}(u) - d_{V^-}(u) + \bClr_0(v) - \one_{U^-}(u) \ge 0
        \right)
        \\
        & \Pr\left(
            d_{W^+}(u) - d_{W^-}(u) + d_{U^+}(u) - d_{U^-}(u) \bClr_0(v) - \one_{U^-}(u) \ge 0
        \right).
    \end{aligned}
    \end{equation*}
    Note that in our conditioned space, the only random variables in the last expression are $d_{W^+}(u) \sim \Bin(|W^+|, p)$ and $d_{W^-}(u) \sim \Bin(|W^-|, p)$.
    Temporarily let $m_u$ to be the fixed part.
    We have
    \begin{equation*}
        m_u \defeq d_{U^+}(u) - d_{U^-}(u) + \bClr_0(v) - \one_{U^-}(u)
        = \pm (\|\mathbf{d}_U\|_\infty + 2)
        = \pm O(\log n).
    \end{equation*}

    By Berry-Esseen, we can replace $X^+ - X^-$ with a Gaussian random variable with mean $\E{X^+ - X^-} = p(|W^+| - |W^-|)$ and variance $\Var{X^+ - X^-} = p(1 - p)|W|$, at the cost of an error
    \begin{equation*}
        C_0 \frac{p(1 - p)[p^2 + (1 - p)^2]|W|}{[p(1 - p)|W|]^{3/2}}
        \le \frac{2C_0}{\sqrt{pn}}.
    \end{equation*}
    We have, by regularity of $U$,
    \begin{equation}  \label{eq:P(a_v)-proof-eps_u-crude}
    \begin{aligned}
        & \Pr(\bClr_1(u) = 1)
        = \Pr\left(
            X^+ - X^- + m_u \ge 0
        \right)
        = \Phi\left(
            \frac{m_u + p(|W^+| - |W^-|)}{\sqrt{p|W|}}
        \right)
        \pm \frac{2C_0}{\sqrt{pn}}
        \\
        & = \frac{1}{2}
        + O\left(
            \frac{m_u + p(|W^+| - |W^-|)}{\sqrt{p|W|}}
        \right)
        \pm \frac{2C_0}{\sqrt{pn}}
        =  \frac{1}{2}
        + O\left(
            \frac{\pm \log n \pm p\sqrt{pn\log n}}{\sqrt{pn}}
        \right)
        \pm \frac{2C_0}{\sqrt{pn}}
        \\
        & = \frac{1}{2} \pm O\left(\frac{\log n}{\sqrt{pn}}\right)
        \quad \implies \ 
        |\eps_u| = O\left(\frac{\log n}{\sqrt{pn}}\right).
    \end{aligned}
    \end{equation}
    This crude bound allows us to estimate $\Var{\bClr_1(v)}$ and the Berry-Esseen error in Eq. \eqref{eq:P(A_v)-proof-esseen1}.
    We have
    \begin{equation}  \label{eq:P(A_v)-proof-esseen1-var}
        \Var{\bClr_1(v)} = |U| - \sum_{u\in U}\eps_u^2
        = |U| \left(
            1 - O\left(\frac{\log^2 n}{pn}\right)
        \right).
    \end{equation}
    Recall that we are allowed $pn \ge C\log^2 n$ for a large enough $C$, so this bound is still meaningful.
    For the error term, temporarily let $\delta = (pn)^{-1/2}\log n$, we have
    \begin{equation}  \label{eq:P(A_v)-proof-esseen1-err}
        \frac{|U| - \sum_{u\in U}\eps_u^4}{
            \left( |U| - \sum_{u\in U}\eps_u^2 \right)^{3/2}
        }
        = \frac{|U|(1 - O(\delta^4))}{|U|^{3/2}(1 - O(\delta^2))^{3/2}}
        = O\left( \frac{1}{\sqrt{pn}} \right).
    \end{equation}
    We keep these bounds in mind for later use.

    Now let us estimate $\sum_{u\in U}\eps_u$ more carefully.
    We expect more cancellations than simply adding the crude bounds in Eq. \eqref{eq:P(a_v)-proof-eps_u-crude}.
    By \cite[Lemma 2.5]{jhk-etal2021}, we have
    \begin{equation*}
    \begin{aligned}
        & \sum_{u\in U} \Pr(\bClr_1(u) = 1)
        = \sum_{u\in U} \Pr\left(
            X^+ - X^- + m_u \ge 0
        \right)
        \\
        & = \sum_{u\in U} \left[
            \Pr(X^+ \ge X^-) + m_u\Pr(X^+ = X^-)
            \pm O\left(\frac{m_u^2}{p|W|}\right)
        \right]
        \\
        & = |U|\Pr(X^+ \ge X^-)
        + \Pr(X^+ = X^-)\sum_{u\in U}m_u
        \pm O\left(\frac{1}{pn}\sum_{u\in U}m_u^2\right).
    \end{aligned}
    \end{equation*}
    Let us simplify the above expression.
    We have
    \begin{equation*}
    \begin{aligned}
        \sum_{u\in U} m_u
        & = \sum_{u\in U} \left(d_{U^+}(u) - d_{U^-}(u)\right)
        + \bClr_0(v)|U| - |U^-|
        = \one_{|U|}^T\mathbf{d}_U
        + \bClr_0(v)|U| - |U^-|
        \\
        & = \bClr_0(v)|U| - \frac{|U|}{2} \pm 4\sqrt{pn\log n},
    \end{aligned}
    \end{equation*}
    where the last estimate is due to $U$ being regular.
    Next, also by regularity,
    \begin{equation*}
    \begin{aligned}
        \sum_{u\in U} m_u^2
        & \le 2\sum_{u\in U} \left[\left(d_{U^+}(u) - d_{U^-}(u)\right)^2 + 4\right]
        = \|\mathbf{d}_U\|_2^2 + 8|U|
        \le 10pn\log n.
    \end{aligned}
    \end{equation*}
    Therefore
    \begin{equation}  \label{eq:P(A_v)-proof-sum-eps_u-temp1}
    \begin{aligned}
        \sum_{u\in U} \Pr(\bClr_1(u) = 1)
        & = |U|\Pr(X^+ \ge X^-)
        + \Pr(X^+ = X^-) \left(
            \bClr_0(v)|U| - \frac{|U|}{2} \pm 4\sqrt{pn\log n}
            \right)
        \pm O(\log n)
        \\
        & = |U|\left(
            \Pr(X^+ \ge X^-) - \frac{\Pr(X^+ = X^-)}{2}
            + \bClr_0(v)\Pr(X^+ = X^-)
        \right)
        \pm O(\log n).
    \end{aligned}
    \end{equation}
    Ideally, we want the first term to simplify $|U|/2$, so that the sum on the left-hand side is close to $|U|/2$, meaning $\sum_{u\in U}\eps_u$ is close to $0$.
    This will work out if $X^+$ and $X^-$ have identical distributions.
    While this is not the case, we can replace them with i.i.d. random variables at some cost.

    Let $Z^+$ and $Z^-$ be two independent copies of $\Bin(n/2 - \lceil pn \rceil, p)$.
    By \cite[Corollary 2.8]{jhk-etal2021}, we can replace $X^\pm$ respectively with $Z^\pm$ in $\Pr(X^+\ge X^-)$ and $\Pr(X^+ = X^-)$ for an error of at most
    \begin{equation*}
        O\left(\frac{p\bigl(||W^+| - (n/2 - \lceil pn \rceil)| \vee ||W^-| - (n/2 - \lceil pn \rceil)|\bigr)}{\sqrt{pn}}\right)
        = O\left(\frac{p\sqrt{pn\log n}}{\sqrt{pn}} \right)
        = O(p\sqrt{\log n}).
    \end{equation*}
    Now by symmetry,
    $
    ~\Pr(Z^+\ge Z^-) - \frac{1}{2}\Pr(Z^+ = Z^-) = \frac{1}{2},
    $
    so Eq. \eqref{eq:P(A_v)-proof-sum-eps_u-temp1} becomes
    \begin{equation*}
    \begin{aligned}
        \sum_{u\in U} \Pr(\bClr_1(u) = 1)
        & = |U|\left[\frac{1}{2} \pm O(p\sqrt{\log n}) + \bClr_0(v)\Pr(Z^+ = Z^-)\right] \pm O(\log n) \\
        & = \frac{|U|}{2} + \bClr_0(v)|U|\Pr(Z^+ = Z^-) \pm O(p|U|\sqrt{\log n}) \pm O(\log n) \\
        & = \frac{|U|}{2} + \bClr_0(v)|U|\Pr(Z^+ = Z^-) \pm O\bigl(\log n\bigr),
    \end{aligned}
    \end{equation*}
    where we use the fact $p|U| = O(p^2n) = O(\sqrt{\log n})$.
    Since $\eps_u = 2\Pr(\bClr_1(u) = 1) - 1$, we have
    \begin{equation}  \label{eq:P(A_v)-proof-sum-eps_v}
        \sum_{u\in U} \eps_u = 2\bClr_0(v)|U|\Pr(Z^+ = Z^-) \pm O\bigl(\log n\bigr).
    \end{equation}
    Note that the first term on the right-hand side is $\Theta(\sqrt{pn})$, so it is indeed the main term.

    Now we can finally estimate $\Pr(A_v)$.
    Plugging Eqs. \eqref{eq:P(A_v)-proof-esseen1-var}, \eqref{eq:P(A_v)-proof-esseen1-err}, \eqref{eq:P(A_v)-proof-sum-eps_v} into Eq. \eqref{eq:P(A_v)-proof-esseen1}, we get
    \begin{equation*}
        \Pr(A_v)
        = \Phi\left(
            \frac{
                D + 2\bClr_0(v)|U|\Pr(Z^+ = Z^-) \pm O\bigl(\log n\bigr)
            }{
                \sqrt{|U|(1 - O(\delta^2))}
            }
        \right)
        \pm O\left(\frac{1}{\sqrt{pn}}\right),
    \end{equation*}
    where $\delta = (pn)^{-1/2}\log n$ as previously defined.
    We analyze the term inside $\Phi(\cdot)$.
    By the regularity of $U$, $|U| = pn(1 \pm O(\delta))$, so we can write this term as
    \begin{equation}  \label{eq:P(A_v)-proof-esseen2}
    \begin{aligned}
        & \frac{D}{\sqrt{|U|(1 - O(\delta^2))}}
        + \frac{
            2\bClr_0(v)\Pr(Z^+ = Z^-)\sqrt{|U|}
        }{\sqrt{1 - O(\delta^2)}}
        \pm \frac{O(\sqrt{\log n})}{\sqrt{|U|(1 - O(\delta^2))}}
        \\
        & = \frac{D}{\sqrt{pn(1 \pm O(\delta))}}
        + 2\bClr_0(v)\Pr(Z^+ = Z^-)\sqrt{pn(1 \pm O(\delta))}
        \pm O\bigl(\delta^2\bigr)
        \\
        & = \frac{D}{\sqrt{pn}}
        + 2\bClr_0(v)\Pr(Z^+ = Z^-)\sqrt{pn}
        \pm O\left(
            \delta \left(
                \frac{D}{\sqrt{pn}} + \Pr(Z^+ = Z^-)\sqrt{pn} + \delta
            \right)
        \right).
    \end{aligned}
    \end{equation}
    Consider the error term.
    By the assumption $D \le 10\sqrt{pn}$, and the facts that $\Pr(Z^+ = Z^-) \le 4(pn)^{-1/2}$ and $\delta < .1$, this error term is just $O(\delta) = O((pn)^{-1/2}\log n)$.
    We get
    \begin{equation}  \label{eq:P(A_v)-proof-esseen3} 
        \Pr(A_v) = \Phi\left(
            \frac{D \pm O(\log n)}{\sqrt{pn}}
            + 2\bClr_0(v)\Pr(Z^+ = Z^-)\sqrt{pn}
        \right) \pm O\left( \frac{1}{\sqrt{pn}} \right).
    \end{equation}
    
    By allowing $D \ge C\log n$ for a sufficient constant $C$, we have
    $D\pm O(\log n) \ge D/2$.
    Eq. \eqref{eq:P(A_v)-proof-esseen3} then implies
    \begin{equation*}
        P(A_v) = \Phi\left(
            \frac{D}{2\sqrt{pn}}
            + 2\bClr_0(v)\Pr(Z^+ = Z^-)\sqrt{pn}
        \right)
        \pm O\left(\frac{1}{\sqrt{pn}}\right),
    \end{equation*}
    which is the desired expression.
    We are not done yet, since this is actually $\Pr(A_v \mid U = N(v), H = G(U))$ for the fixed regular pair $(U, H)$ that we have been using.
    However, since the total probability of irregular pairs is $O(n^{-2})$, which gets absorbed by the error term $O((pn)^{-1/2})$ above, this does not affect $\Pr(A_v)$ in the general case.
    The proof is complete.
\end{proof}

\subsubsection{Step 2: Upper bounding $\Var{|A|}$}

Below is the main claim

\begin{claim}  \label{claim:almost-pos-set-var} 
    Let $A$ be the set of $D$-almost positive vertices.
    Let $B > 0$ be an arbitrary constant.
    There are universal constants $C_1$, $C_2$ (large enough) such that if $pn \ge C_1\log^2 n$ and $B\sqrt{pn} \ge D \ge C_2\log n$, then
    \begin{equation*}
        \Var{|A|} = O\left(\frac{n\log^2 n}{p} \right).
    \end{equation*}
\end{claim}

The key idea is to write
\begin{equation*}
    \Var{|A|} = \sum_{v\in V} \Var{\one_{A_v}} + \sum_{v_1\neq v_2}\Cov{\one_{A_{v_1}}, \one_{A_{v_2}}}
\end{equation*}
Since the first part is at most $n/4$, it suffices to consider the covariances.
Fix a distinct pair $v_1, v_2\in V$, we have the claim

\begin{claim}  \label{claim:Cov(A_v1, A_v2)}
    Under the settings of Claim \ref{claim:almost-pos-set-var}, we have
    \begin{equation*}
        \left| \Cov{\one_{A_{v_1}}, \one_{A_{v_2}}} \right|
        = O\left( \frac{\log^2 n}{pn} \right).
    \end{equation*}
\end{claim}

Suppose this holds, we can prove Claim \ref{claim:almost-pos-set-var}.
\begin{proof}[Proof of Claim \ref{claim:almost-pos-set-var}]
    By the decomposition above, we have
    \begin{equation*}
        \Var{|A|} \le \frac{n}{4} + \binom{n}{2} \cdot
        O\left( \frac{\log^2 n}{pn} \right)
        = O\left( \frac{n\log^2 n}{p} \right).
    \end{equation*}
    The proof is complete.
\end{proof}

Now we move to the main technical challenge.
Similarly to the proof of Claim \ref{claim:P(A_v)}, we want to prove Claim \ref{claim:Cov(A_v1, A_v2)} in an ``ideal'' environment, namely where the statistics of $N(v_1)$ and $N(v_2)$ are close to their expectations.

This time we want to expose everything inside both $N(v_1)$ and $N(v_2)$.
Fix $U\subset V$ with partition $U = U_1 \cup U_2 \cup U_3$ and a subgraph $H$ over $U$, and condition $G$ on the events $U_1 = N(v_1) \setminus N(v_2)$, $U_2 = N(v_2) \setminus N(v_2)$, $U_3 = N(v_1) \cap N(v_2)$ (all excluding $v_1$ and $v_2$) and $H = G(U)$.

We introduce some extra notation:
\begin{itemize}
    \item $U_{13} \defeq U_1 \cup U_3 = N(v_1) \setminus \{v_2\}$ and $U_{23} \defeq U_2 \cup U_3 = N(v_2) \setminus \{v_2\}$ in our conditioned space.

    \item $U^+ \defeq U\cap \bR_0$ and $U^- \defeq U\cap \bB_0$.
    Define the same ``Red'' and ``Blue'' versions for the sets defined above.

    \item $W = (V\setminus \{v_1, v_2\}) \setminus U$.
    Define $W^+$ and $W^-$ similarly to the above.

    \item $V^+ = U^+ \cup W^+$ and $V^- = U^- \cup W^-$.

    \item The tuple $(U_1, U_2, U_3, H)$ is \emph{regular} if the following conditions hold.
    \begin{equation*}
    \begin{aligned}
        |U_{13}^+|, |U_{13}^-|, |U_{23}^+|, |U_{23}^-| = \frac{pn}{2} \pm 3\sqrt{pn\log n}, \quad
		|U_3^+|, |U_3^-| \le \log n,
        \\
        \|\mathbf{d}_U\|_\infty \le \log n,
        \quad
        \|\mathbf{d}_U\|_2 \le 3\sqrt{pn\log n},
        \quad
        \left| \one_{|U_{13}}^T \mathbf{d}_{U_{13}} \right|
        \vee \left| \one_{|U_{23}}^T \mathbf{d}_{U_{23}} \right|
        \le 5\log n,
    \end{aligned}
    \end{equation*}
    where $\mathbf{d}_U = [d_{U^+}(u, H) - d_{U^-}(u, H)]_{u\in U} \in \R^{|U|}$ is the previously defined vector
    and $\mathbf{d}_{U_{13}} \in \R^{|U_{13}|}$ and $\mathbf{d}_{U_{23}} \in \R^{|U_{23}|}$ are its restrictions to $U_{13}$ and $U_{23}$ respectively.
\end{itemize}

Again, simple Chernoff bounds show that $U_1, U_2, U_3$ and $H$ are regular with probability $1 - O(n^{-2})$.
Now we can proceed to the formal proof.

\begin{proof}[Proof of Claim \ref{claim:Cov(A_v1, A_v2)}]
    Let $h_i = \E{\one_{A_{v_i}}}$ without any conditioning.
    Fix disjoint $U_1, U_2, U_3$ and subgraph $H$ over $U = U_1 \cup U_2 \cup U_3$ such that they are regular.
    Condition $G$ on the events $(N(v_1) \setminus \{v_2\}) \setminus N(v_2) = U_1$, 
    $(N(v_2) \setminus \{v_1\}) \setminus N(v_1) = U_2$,
    $(N(v_1) \setminus \{v_2\}) \cap (N(v_2) \setminus \{v_1\}) = U_3$,
    and $H = G(U)$.
    The covariance on this conditioned probability is off from the real covariance by an error of just $O(n^{-2})$.
    Until noted otherwise, we omit the events from the probability.

    We first rewrite the covariance as
    \begin{equation}  \label{eq:Cov(A_v1, A_v2)-proof-cov1}
        \Cov{\one_{A_{v_1}}, \one_{A_{v_2}}}
        = \E{(\one_{A_{v_1}} - h_1)(\one_{A_{v_2}} - h_2)}
    \end{equation}
    We want to split this into a product, but $A_{v_1}$ and $A_{v_2}$ are not independent.
    However, their only shared dependency is on the first day colors of their joint neighbors, $U_3$.
    
    Fix a set of first day colors $\{j_u\}_{u\in U_3}$ for these vertices.
    Condition the space further on the event $\mathcal{J} \defeq \{\forall u\in U_3, \ \bClr_1(u) = j_u\}$.
    This does not affect $\bClr_1(u)$ for the other $u\in U$ as they are all independent.
    We can rewrite the conditioned Eq. \eqref{eq:Cov(A_v1, A_v2)-proof-cov1} as
    \begin{equation*}
        \Cov{\one_{A_{v_1}}, \one_{A_{v_2}} \mid \mathcal{J}}
        = \left(
            \E{\one_{A_{v_1}} \mid \mathcal{J}} - h_1
        \right)
        \cdot \left(
            \E{\one_{A_{v_2}} \mid \mathcal{J}} - h_2
        \right).
    \end{equation*}
    Consider the factor $\E{\one_{A_{v_1}} \mid \mathcal{J}} - h_1$.
    Let $H_1(\mathcal{J}) = \E{\one_{A_{v_1}} \mid \mathcal{J}}$ for convenience.
    Let  $j \defeq \sum_{u\in U_3} j_u$.
    Reuse the computation leading Eq. \eqref{eq:P(A_v)-proof-esseen1} to get
    \begin{equation}  \label{eq:Cov(A_v1, A_v2)-proof-j-esseen1}
        H_1(\mathcal{J})
        = \Phi\left(
            \frac{
                D + \bClr_1(v_2)\one_{v_1\adj v_2}
                +  j + \sum_{u\in U_1} \eps_u
            }{
                (|U_1| - \sum_{u\in U_1}\eps_u^2)^{1/2}
            }
        \right)
        \pm O\left(
            \frac{
                |U_1| - \sum_{u\in U_1}\eps_u^4
            }{
                (|U_1| - \sum_{u\in U_1}\eps_u^2)^{3/2}
            }
        \right).
    \end{equation}
    In this space conditioned on $\mathcal{J}$, the variables $\eps_u = \E{\bClr_1(u)}$ for $u\in U_1$ are not affected since they are independent from $\{\bClr_1(u)\}_{u\in U_3}$.
    Note that this expression only depends on $j$ rather the specific values in $\{\bClr_1(u)\}_{u\in U_3}$.

    Fix another set of colors $\{k_u\}_{u\in U_3}$ and consider the event $\mathcal{K} = \{\forall u\in U_3, \ \bClr_1(u) = k_u\}$.
    Define $H_1(\mathcal{K})$ analogously to $H_1(\mathcal{J})$.
    The expression for $H_1(\mathcal{K})$ is almost identical to Eq. \eqref{eq:Cov(A_v1, A_v2)-proof-j-esseen1}, with $k = \sum_{u\in U_3} k_u$ replacing $j$.
    We then have
    \begin{equation}  \label{eq:Cov(A_v1, A_v2)-proof-j-k-diff}
        \left| H_1(\mathcal{J}) - H_1(\mathcal{K}) \right|
        = O\left( \frac{|j - k|}{(|U| - \sum_{u\in U_1}\eps_u^2)^{1/2}}
        \right)
        = O\left( \frac{|U_3|}{\sqrt{pn}}
        \right)
        = O\left( \frac{\log n}{\sqrt{pn}}
        \right),
    \end{equation}
    where the last two bounds are due to regularity and the analysis of $\eps_u$ in Eq. \eqref{eq:P(a_v)-proof-eps_u-crude}.

    This is an innovation from the proof in \cite{jhk-etal2021}.
    This allows us to indirectly analyze $H_1(\mathcal{J}) - h_1$ without directly computing $H_1(\mathcal{J})$.
    By showing that this expression does not change much when $\mathcal{J}$ changes, $H_1(\mathcal{J}) = \Pr(A_{v_1} \mid \mathcal{J})$ must be close to the same probability without conditioning on $\mathcal{J}$.
    
    Let $H_1 \defeq \Pr(A_{v_1} \mid U_1, U_2, U_3, H)$ (we write the background events explicitly here to differentiate this with $h_1$, which is $\Pr(A_{v_1})$ without any condition).
    We have, by Eq. \eqref{eq:Cov(A_v1, A_v2)-proof-j-k-diff},
    \begin{equation}  \label{eq:Cov(A_v1, A_v2)-proof-H_1(j)-H_1-diff}
        \left| H_1(\mathcal{J}) - H_1 \right|
        = O\left( \frac{\log n}{\sqrt{pn}} \right),
    \end{equation}

    We are halfway through bounding $H_1(\mathcal{J}) - h_1$.
    It remains to bound $H_1 - h_1$.
    The idea is not too different from the previous step.
    We write $H_1 = H_1(U_1, U_2, U_3, H)$, since we will look at the same variable conditioned on another regular tuple.
    First off, by almost analogous calculations to the steps leading up Eq. \eqref{eq:P(A_v)-proof-esseen3} in Section \ref{sec:almost-pos-set-exp}, we have
    \begin{equation*}
        H_1(U_1, U_2, U_3, H)
        = \Phi\left(
            \frac{D \pm O(\log n)}{\sqrt{pn}}
            + 2\bClr_0(v_1)\Pr(Z^+ = Z^-)\sqrt{pn}
        \right) \pm O\left( \frac{1}{\sqrt{pn}} \right).
    \end{equation*}
    Now fix another regular tuple $(U'_1, U'_2, U'_3, H')$.
    The same calculation applies, so we are left with
    \begin{equation*}
        \left| H_1(U_1, U_2, U_3, H)
        - H_1(U'_1, U'_2, U'_3, H') \right|
        = O\left( \frac{\log n}{\sqrt{pn}} \right).
    \end{equation*}
    The weights of the irregular tuples are just $O(n^{-2})$, which is asorbed by the error bound above, so we have
    \begin{equation}  \label{eq:Cov(A_v1, A_v2)-proof-H_1-h_1-diff}
        \left| H_1(U_1, U_2, U_3, H) - h_1 \right|
        = O\left( \frac{\log n}{\sqrt{pn}} \right).
    \end{equation}

    Back the to probability space conditioned on $(U_1, U_2, U_3, H)$.
    Now we can finally bound $H_1 - h_1$.
    Combining Eqs. \eqref{eq:Cov(A_v1, A_v2)-proof-H_1(j)-H_1-diff} and \eqref{eq:Cov(A_v1, A_v2)-proof-H_1-h_1-diff}, we get
    \begin{equation*}
        \left| H_1 - h_1 \right| = O\left( \frac{\log n}{\sqrt{pn}} \right).
    \end{equation*}
    An analogous calculation gives the same bound for $H_2 - h_2$.
    This implies
    \begin{equation*}
        \left| 
            \Cov{\one_{A_{v_1}}, \one_{A_{v_2}} \mid U_1, U_2, U_3, H}
        \right|
        = O\left( \frac{\log^2 n}{pn} \right).
    \end{equation*}
    The proof is complete, since the total weight of the irregular cases is $O(n^{-2})$, which is absorbed.
\end{proof}

\subsubsection{Proof of Lemma \ref{lem:almost-pos-v}}

Recall that $|A|$ is being used to denote the set of $D$-almost positive vertices.
Let us recall the main claims about $|A|$ from the previous two steps.
Suppose $pn \ge C_1\log^2 n$ and $B\sqrt{pn} \ge D \ge C_2\log n$ for arbitrary $B > 0$ and sufficient constants $C_1$ and $C_2$ that make Claims \ref{claim:almost-pos-set-exp} and \ref{claim:almost-pos-set-var} hold, we have
\begin{equation*}
    \E{|A|} \ge \frac{n}{2} + 2c_BD\sqrt{\frac{n}{p}}
    \ \text{ and } \
    \sqrt{\Var{|A|}} \le C\log n\sqrt{\frac{n}{p}},
\end{equation*}
for constants $c_B > 0$ (small enough) only depending on $B$ and universal $C > 0$ (large enough).

\begin{proof}[Proof of Lemma \ref{lem:almost-pos-v}]
    We continue from the argument above.
    By Chebyshev's inequality,
    \begin{equation*}
        \Pr\left(|A| \le \frac{n}{2} + c_BD\sqrt{\frac{n}{p}} \right)
        \le \frac{\Var{|A|}}{(c_BD\sqrt{n/p})^2}
        \le \frac{C^2\log^2 n}{c_B^2D^2} = O_B\left(\frac{\log^2 n}{D^2} \right).
    \end{equation*}
    The proof is complete.
\end{proof}

\section{Preliminaries bounds on binomial random variables}  \label{sec:binom-bound}

In this section we describe several useful inequalities to bound binomial random variables.
The first one is the usual Chernoff bound \cite{hoeffding1963}.

\begin{theorem}[Chernoff bound]  \label{thm:chernoff}
	Let $X\sim \Bin(n, p)$, and $\eps > 0$. We have
	\begin{equation*}
		\begin{aligned}
			\Pr\left(X \ge (p + \eps)n\right) \le e^{-D(p + \eps \| p)n} & \quad \text{ if } \eps < 1 - p, \\
			\Pr\left(X \le (p - \eps)n\right) \le e^{-D(p - \eps \| p)n} & \quad \text{ if } \eps < p,
		\end{aligned}
	\end{equation*}
	where $D(x \| y) \defeq x\log\dfrac{x}{y} + (1 - x)\log\dfrac{1 - x}{1 - y}$ for $x, y\in (0, 1)$.
\end{theorem}

In this paper, we mainly use the following corollary of the Chernoff bound:

\begin{corollary}  \label{cor:chernoff}
	Let $X\sim \Bin(n, p)$ and $t > 0$. We have
	\begin{equation*}
		\Pr\left(X = pn \pm t\sqrt{pn}\right) \ge 1 - 2\exp\left(-\frac{t^2}{2(1 + \frac{t}{3\sqrt{pn}})} \right).
	\end{equation*}
\end{corollary}

\begin{proof}
	Let $\eps = tp^{1/2}n^{-1/2}$, so that $pn + t\sqrt{pn} = (p + \eps)n$.
	If $\eps \ge 1 - p$, the inequality trivially holds.
	If $\eps < 1 - p$, we have by Theorem \ref{thm:chernoff}:
	\begin{equation}  \label{eq:chernoff-temp1}
		\Pr(X > pn + t\sqrt{pn}) \le \exp\left[-D\Bigl(p + t\frac{\sqrt{p}}{\sqrt{n}} ~\Big\|~ p\Bigr)n\right].
	\end{equation}
	It is a well-known inequality that $D(x \| y) \ge 3(x - y)^2(2x + 4y)^{-1}$ for each $x, y \ge 0$.
	Thus
	\begin{equation*}
		D\Bigl(p + t\frac{\sqrt{p}}{\sqrt{n}} ~\Big\|~ p\Bigr)n
		\ge \frac{3\bigl(tp^{1/2}n^{-1/2}\bigr)^2}{2(3p + tp^{1/2}n^{-1/2})}n
		= \frac{t^2}{2(1 + \frac{t}{3\sqrt{pn}})}.
	\end{equation*}
	Plugging in Eq. \eqref{eq:chernoff-temp1}, we get the desired bound.
	
	Consider the event $X < pn - t\sqrt{pn} = (p - \eps)n$.
	If $\eps \ge p$, the inequality trivially holds.
	If $\eps < p$, i.e. $t < \sqrt{pn}$, we have by Theorem \ref{thm:chernoff}:
	\begin{equation}  \label{eq:chernoff-temp2}
		\Pr(X < pn - t\sqrt{pn}) \le \exp\left[-D\Bigl(p - t\frac{\sqrt{p}}{\sqrt{n}} ~\Big\|~ p\Bigr)n\right].
	\end{equation}
	We have
	\begin{equation*}
		D\Bigl(p - t\frac{\sqrt{p}}{\sqrt{n}} ~\Big\|~ p\Bigr)n
		\ge \frac{3\bigl(tp^{1/2}n^{-1/2}\bigr)^2}{2(3p - tp^{1/2}n^{-1/2})}n
		= \frac{t^2}{2(1 - \frac{t}{3\sqrt{pn}})}
		\ge \frac{t^2}{2(1 + \frac{t}{3\sqrt{pn}})},
	\end{equation*}
	where the last inequality is due to $0 < t < \sqrt{pn}$.
	Plugging in Eq. \eqref{eq:chernoff-temp2}, we get the desired bound.
\end{proof}

We quickly record another bound useful for binomial variables with small expectations:
\begin{lemma}  \label{lem:poisson-bound}
	For any $X\sim \Bin(n, p)$ and $t \ge epn$, we have:
	$ \Pr(X > t) \le 2\left(\frac{epn}{t}\right)^t
	$.
\end{lemma}

\begin{proof}
	For each integer $k\in [0, n]$,
	$ \Pr(X = k) = \binom{n}{k}p^k(1 - p)^{n - k} \le \frac{n^k}{k!}p^k $.
	Let $k = \lceil t \rceil$, we have
	\begin{equation*}
		\begin{aligned}
			\Pr(X > t)
			& \le \sum_{l=k}^\infty \frac{(pn)^l}{l!}
			= \frac{(pn)^k}{k!} \sum_{i=0}^\infty \frac{(pn)^i}{(i + k)(i + k - 1)\ldots (k + 1)}
			\le \frac{(pn)^k}{k!}  \sum_{i=0}^\infty \frac{(pn)^i}{(k + 1)^i} \\
			& = \frac{(pn)^k}{k!} \sum_{i=0}^\infty e^{-i} < 2\frac{(pn)^k}{k!}
			< 2\left(\frac{epn}{k}\right)^k < 2\left(\frac{epn}{t}\right)^t,
		\end{aligned}
	\end{equation*}
	where the second to last inequality is the usual Stirling bound \cite{stirling}, and the last is due to decreasing property of the expression $(a/t)^t$ for $t \ge a$.
\end{proof}

\bibliographystyle{plain}
\bibliography{references}

\end{document}